\documentclass[a4paper, 11pt]{article}

\usepackage[
            includefoot,  %Uncomment to put page number above margin
            marginparwidth=0in,     % Length of section titles
            marginparsep=0in,       % Space between titles and text
            margin=1.45in,               % 1 inch margins
            includemp]{geometry}

\usepackage{bbm}
\usepackage{float}
\usepackage{graphicx}                  % derni\`ere \'etant la langue principale
\usepackage{amssymb}
\usepackage{amsfonts}
\RequirePackage{amsmath}
\RequirePackage{amsthm}

\usepackage{color}

\usepackage{fancyhdr}

\newtheorem{thm}{Theorem}[section]

\newtheorem{lem}[thm]{Lemma}
\newtheorem{rque}{Remark} [section]
\newtheorem{cor}[thm]{Corollary}

\DeclareMathOperator{\Hess}{Hess}

\DeclareMathOperator{\Var}{Var}

\DeclareMathOperator{\Ent}{Ent}

\newcommand{\BE}{Bakry-\'Emery }
\newcommand{\R}{\mathbb{R}}

\begin{document}

\title{Stability of the \BE theorem on $\R^n$}
\date{\today}

\author{Thomas~A.~Courtade$^*$ and Max Fathi$^{\dagger}$\\
~\\
$^{*}$UC Berkeley, Department of Electrical Engineering and Computer Sciences\\
$^{\dagger}$CNRS \& Universit\'e Paul Sabatier, Institut de Math\'ematiques de Toulouse\\
}

\maketitle

\begin{abstract}
We prove stability estimates for the \BE bound on Poincar\'e and logarithmic Sobolev constants of uniformly log-concave measures. In particular, we improve the quantitative bound in a result of De Philippis and Figalli asserting that if a $1$-uniformly log-concave measure has almost the same Poincar\'e constant as the standard Gaussian measure, then it almost splits off a Gaussian factor, and establish similar new results for logarithmic Sobolev inequalities.    As a consequence, we obtain dimension-free stability estimates for Gaussian concentration of Lipschitz functions.   The proofs are based on Stein's method,   optimal transport, and an approximate integration by parts identity relating measures and approximate optimizers in the associated functional inequality. 
\end{abstract}

\section{Introduction and main results}

The purpose of this work is to establish stability estimates for the \BE theorem, which states that the sharp constant for various functional inequalities for uniformly log-concave measures must be better than the sharp constant for the standard Gaussian measure. We shall focus on two main inequalities: the Poincar\'e inequality and the logarithmic Sobolev inequality. 

\subsection{Poincar\'e inequality}

A probability measure on $\R^n$ is said to satisfy a Poincar\'e inequality with constant $C$ if for any smooth test function $f$, its variance satisfies the bound
$$\Var_{\mu}(f) \leq C\int{|\nabla f|^2d\mu}.$$
The smallest possible constant in this inequality is called the Poincar\'e constant of $\mu$, which we shall denote by $C_P(\mu)$. Such inequalities play an important role in several areas of analysis, probability and statistics, such as concentration of measure, rates of convergence for stochastic dynamics and analysis of PDEs. This constant is also the inverse of the spectral gap of the Fokker-Planck (or overdamped Langevin) dynamic associated with $\mu$. A large class of probability measures satisfy such an inequality. In particular, if a probability measure is more log-concave than the standard Gaussian measure (that is, $\mu = e^{-V}dx$ with $\Hess V \geq \operatorname{I}_n$), then $C_P(\mu) \leq 1$. This result can be viewed as a consequence of the Brascamp-Lieb inequality \cite{BL76} or the \BE theory \cite{BE85}. 

More recently, Cheng and Zhou \cite{CZ} proved a rigidity property for the \BE theorem: if such a probability measure has its Poincar\'e constant equal to one, then it must be a product measure, with one of the factors being a Gaussian measure of unit variance. They also obtain a rigidity result in the more general setting of complete metric-measure spaces with positive Ricci curvature. See also \cite{HJS} for a weaker form of this rigidity in $\R^n$, and \cite{CL87} for rigidity in a different class of measures (and \cite{CFP18} for a corresponding stability estimate). 

The convexity condition we shall assume here is a particular case of the \BE curvature-dimension condition, itself a generalization of Ricci curvature lower bounds. Splitting theorems for manifolds satisfying a curvature bound and a geometric condition have been the topic of some interest, going back to work of Cheeger and Gromoll \cite{CG71, CC96}. More recently, rigidity and stability for a related (and stronger) isoperimetric inequality has been established \cite{CMM17} under the stronger curvature-dimension condition with finite dimension, using completely different techniques. 

Poincar\'e inequalities can be viewed as estimates on the smallest eigenvalue of the diffusion operator $-\Delta + \nabla V \cdot \nabla$. Stability for other spectral problems have been considered, such as Poincar\'e inequalities on bounded domains \cite{BDPV1, BDP} and a lower bound on the spectrum of Schr\"odinger operators \cite{CFL}, respectively with applications in shape optimization and quantum mechanics. 

The first main result of the present work is to improve the quantitative bounds in the following result of De Philippis and Figalli \cite{DPF}, which establishes a strong form of quantitative stability for the Poincar\'e constant for uniformly log-concave measures. 

\begin{thm} \label{thm_dpf}
Let $\mu = e^{-V}dx$ be a probability measure with $\Hess V \geq \operatorname{I}_n$, and assume that there exists $k$ functions $u_i \in H^1(\mu)$, $k \leq n$, such that for any $i \in \{1,..,k\}$ we have
$$\int{u_id\mu} = 0; \hspace{5mm} \int{u_i^2d\mu} = 1; \hspace{5mm} \int{\nabla u_i \cdot \nabla u_j d\mu} = 0, ~~~ \forall j\neq i$$
and 
$$\int{|\nabla u_i|^2d\mu} \leq 1 + \epsilon$$
for some $\epsilon \geq 0$. Then for any $\theta > 0$ there exists $C(n, \theta)$, a subspace $\mathcal{V}\subset \R^n$ with $\dim(\mathcal{V})=k$, and a vector  $p \in \mathcal{V}$  such that
$$W_1(\mu, \gamma_{p,\mathcal{V}} \otimes \bar{\mu}) \leq C(n,\theta)|\log \epsilon|^{-1/4 + \theta},$$
where $\gamma_{p,\mathcal{V}}$ is the standard Gaussian measure on $\mathcal{V}$ with barycenter $p$, and $\bar{\mu}$ is the marginal distribution of $\mu$ on $\mathcal{V}^{\bot}$, which enjoys the same convexity property as $\mu$.
\end{thm}
%and a function $W : \R^{n-k} \longrightarrow \R$ satisfying $\Hess W \geq \operatorname{I}_{n-k}$

In this statement, $W_1$ stands for the classical $L^1$ Kantorovitch-Wasserstein distance \cite{Vill03}, and $H^1(\mu) := \{f; \int{(|f|^2 + |\nabla f|^2)d\mu} < \infty\}$ is a weighted Sobolev space with respect to $\mu$.  

We shall obtain the following improvement: 

\begin{thm} \label{thm_improved_dpf}
Under the same assumptions as in Theorem \ref{thm_dpf}, we have
$$W_1(\mu, \gamma_{p,\mathcal{V}}\otimes \bar{\mu}) \leq Ck^{3/2}\sqrt{\epsilon}.$$
In fact, we may take $C=18 \sqrt{2} < 26$.
\end{thm}

Beyond the improved dependence on $\epsilon$, the fact that our bound depends on $k$ and not on $n$ is useful for high-dimensional situations. 

The proof of \cite{DPF} is based on a stability version of Caffarelli's contraction theorem \cite{Caf00}, which is a regularity estimate on the nonlinear Monge-Amp\`ere PDE. To obtain the improved bound, we shall rely instead on Stein's method \cite{Ste72, Ste86}, which is a way of quantifying distances between probability measures using well-chosen integration by parts formulas. See \cite{Ros11} for an introduction to the topic. The main reason why this allows us to obtain better estimates is that this proof mostly remains at a linear level, instead of relying on nonlinear tools as in \cite{DPF}. The other main tool in the proof is that the test functions in the assumptions of the theorem can be viewed as approximate minimizers in a variational problem, which then give rise to an approximate Euler-Lagrange equation (up to reminder terms of order $\sqrt{\epsilon}$), which takes the form of an integration by parts formula. See Section \ref{strat_proof} for an overview of the strategy of proof. 
 
\begin{rque} When $k=n$, it is possible to improve the topology, and get an estimate of order $\sqrt{\epsilon}$ in the stronger $W_2$ distance, using results from \cite{LNP15, CFP18}. We do not know how to get a $W_2$ estimate when $k < n$, due to regularity issues for the Poisson equation we shall make use of in the proof. 
\end{rque}

We do not know if the bound is optimal. Testing on Gaussian measures with variance $1-\epsilon$ shows that the optimal rate cannot be better than $\epsilon$ (instead of $\sqrt{\epsilon}$). See also Remark 1.4 in \cite{DPF} for computations in dimension one in a related problem which suggest the sharp rate could be $\epsilon$.

Our main result has the following immediate corollary, which can be viewed as a dimension-free improvement of the \BE theorem.  

\begin{cor}\label{cor:1dimBE}
Let $\mu = e^{-V}dx$ be a probability measure with $\Hess V \geq \operatorname{I}_n$. There is a direction $\sigma \in \mathbb{S}^{n-1}$ and a vector $p \in \operatorname{Span}(\sigma)$ such that 
$$W_1(\mu, \gamma_{p,\sigma} \otimes \bar{\mu}) \leq C\sqrt{C_P(\mu)^{-1}-1},$$
where $\gamma_{p,\sigma}$ is the standard Gaussian measure on $\operatorname{Span}(\sigma)$ with barycenter $p$, and $\bar{\mu}$ is the marginal distribution of $\mu$ on $\operatorname{Span}(\sigma)^{\bot}$. 
\end{cor}

%exists a measure $\bar{\mu}$ on $\R^{n-1}$ satisfying the same convexity assumption, and some $p \in \R$ such that

This corollary follows from Theorem \ref{thm_improved_dpf} since there must be a function $u$ satisfying the assumptions of that theorem for any $\epsilon > C_P(\mu)^{-1}-1$ , by definition of the Poincar\'e constant. We make use here of the fact that the bound in Theorem \ref{thm_improved_dpf} depends on $k$ and not $n$, unlike Theorem \ref{thm_dpf}, to get a dimension-free estimate.    A noteworthy consequence is the following  refinement of the classical dimension-free concentration bound  $\Var_{\mu}(F)\leq 1$ for 1-Lipschitz $F$.
\begin{cor} 
Let the notation of Corollary \ref{cor:1dimBE} prevail.  For any $1$-Lipschitz $F:\R^n \longrightarrow \R$, there exists 
a direction $ \sigma(F) \in \mathbb{S}^{n-1}$ and a vector $  p(F)\in \operatorname{Span}(\sigma)$ such that 
$$
W_1(\mu, \gamma_{p,\sigma} \otimes \bar{\mu}) \leq C \sqrt{  \Var_{\mu}(F)^{-1}-1}.
$$
%In particular, if $\Var_{\mu}(F)$ is very close to 1, then $\mu$ must approximately split off a Gaussian factor.
 \end{cor}

At this point, one might wonder if the convexity assumption is necessary. It cannot simply be dropped: if one looks at a general measure, its Poincar\'e constant may be equal to one, for example by rescaling an arbitrary (but nice) measure to enforce this, in which case there exists a function $u$ satisfying the assumptions, and in general there will not be a Gaussian factor. However, the convexity assumption will mainly be used to ensure the functions $u_i$ are close to coordinate functions, in a suitable basis of $\R^n$, and hence can be dropped if we assume extra knowledge on second moments. This leads to the following result, with improved dependence on $k$: 

\begin{thm} \label{thm_coord}
Assume  $C_P(\mu) \leq 1$, and that there exists an orthonormal family $e_1,..,e_k$ of $\R^n$ such that
$\Var_{\mu}(x \cdot e_i) \geq (1+\epsilon)^{-1}$, for  each $i\leq k$. Then there exists $p\in \mathcal{V} = \operatorname{Span}(e_1, \dots, e_k)$ such that 
%there exists a measure $\bar{\mu}$ on $\R^{n-k}$ and a vector $p\in \R^k$ satisfying
$$W_1(\mu, \gamma_{p,\mathcal{V}} \otimes \bar{\mu}) \leq k  \sqrt{ \pi {\epsilon} },$$
where the measures  $\gamma_{p,\mathcal{V}}$ and   $\bar{\mu}$ are as defined in Theorem \ref{thm_dpf}. 
\end{thm}

%Taking the case $k=1$, and defining $\vartheta := {\sup_{{\sigma \in \mathbb{S}^{n-1}}}} \hspace{1mm} \Var_{\mu}(x \cdot \sigma ) \leq C_P(\mu)\leq 1$, a simple consequence is the following:
%$$W_1(\mu, \gamma_{1,p} \otimes \bar{\mu}) \leq  \sqrt{\pi\frac{ 1-\vartheta}{\vartheta} }.$$

\subsection{Logarithmic Sobolev inequality}

According to the \BE theorem \cite{BE85}, probability measures that are more log-concave than the standard Gaussian measure satisfy the logarithmic Sobolev inequality (LSI)
\begin{equation} \label{eq:LSImu}
\Ent_{\mu}(f^2) \leq 2C_{\mathrm{LSI}}(\mu)\int{|\nabla f|^2d\mu}; \hspace{5mm} C_{\mathrm{LSI}}(\mu) \leq 1
\end{equation}
where $\Ent_{\mu}(f^2) = \int{f^2 \log f^2 d\mu} - \left(\int{f^2d\mu}\right) \log \int{f^2d\mu}$, and $C_{\mathrm{LSI}}(\mu)$ stands for the sharpest possible constant for $\mu$ in this inequality. This functional inequality, originally introduced by Gross \cite{Gro75}, is strictly stronger than the Poincar\'e inequality, and the constant $1$ is sharp for the standard Gaussian measure. Moreover, Carlen  \cite{Car91} showed that for the Gaussian measure equality holds in the LSI if and only if the function $f$ is of the form $f(x) = Ce^{p \cdot x}$ for some vector $p \in \R^n$. The LSI is used to derive Gaussian concentration inequalities, as well as estimates on the rate of convergence to equilibrium for certain stochastic processes. We refer to \cite{BGL14} for background on this inequality and its applications. 

We study stability for the bound on the logarithmic Sobolev constant. Our second main result is the following estimate, showing that if $C_{\mathrm{LSI}}(\mu)$ is close to one, then  $\mu$ still approximately splits off a Gaussian factor, provided the approximate optimizer satisfies regularity assumptions.

\begin{thm}\label{thm:QstableLSI}
Consider a probability measure  $\mu= e^{-V}dx$ on $\R^n$   satisfying $\Hess V \geq \mathrm{I}_n$.  Let $u:\R^n \longrightarrow \R$ be a nonnegative function such that $\log u$ is $\lambda$-Lipschitz and  $\int{u^2d\mu} = 1$.  There exists a constant $C(\lambda)$, depending only on $\lambda$,  such that if 
\begin{align}
\Ent_{\mu}(u^2) \geq 2(1-\epsilon)\int{|\nabla u |^2d\mu}\label{eqn:fApproxExLSI}
\end{align}
for some $\epsilon \geq 0$, then there is a direction $\sigma \in \mathbb{S}^{n-1}$ for which 
\begin{align}
W_1(\mu, \gamma_{b,\sigma} \otimes \bar{\mu}) \leq C(\lambda)\left(\int{|\nabla u |^2d\mu}  \right)^{-1/2}\sqrt{\epsilon},\label{W1boundLSI}
\end{align}
where $\gamma_{b,\sigma}$ denotes the standard Gaussian  measure on $\operatorname{Span}(\sigma)$ with barycenter $b=\sigma \int x\cdot \sigma \, d\mu$, and $\bar{\mu}$ is the marginal distribution of $\mu$ on $\operatorname{Span}(\sigma)^{\bot}$, which enjoys the same convexity property as $\mu$.
\end{thm}

The constant  $C(\lambda)$ can, in principle, be made explicit.  However, its expression would be quite  complicated and our arguments make no attempt to optimize it,  so we   do not attempt to do so. Note that we should expect the bound to get worse if $\int{|\nabla u|^2d\mu}$ is small, since if $u$ was constant it would be a trivial minimizer of the LSI, no matter what $\mu$ would be, so the bound must rule out that situation in some way. Up to the regularity assumption that $\log u$ is Lipschitz, existence of such a function is a weaker assumption than the assumption of Theorem \ref{thm_improved_dpf}, since $C_{\mathrm{LSI}}(\mu) \geq C_P(\mu)$.  

Like Theorem \ref{thm_improved_dpf}, it is possible to give a version of Theorem \ref{thm:QstableLSI} with $k$ orthogonal minimizers, in the sense that $\mu$ approximately splits off a $k$-dimensional factor, provided $\int{\nabla \log u_i \cdot \nabla  \log u_j d\mu} = 0$ for approximate minimizers $(u_i)_{i\leq k}$. The constant $C$ would depend on $k$, but not on the ambient dimension. 

\begin{rque}
The stipulation that $\int{u^2d\mu} = 1$ is for made convenience, and comes without loss of generality.  Indeed, the LSI is homogenous of degree 2, so rescaling $u \longrightarrow \alpha u$ for $\alpha \in \R$  affects neither $\epsilon$-optimality in the sense of \eqref{eqn:fApproxExLSI}, nor the $\lambda$-Lipschitz property assumed of $\log u$.  Further, the assumed nonnegativity of $u$ is also for convenience, and comes without loss of generality since the log-lipschitz assumption already enforces it to have constant sign. 
\end{rque}

\begin{rque}\label{Rmk_utdmu}
Theorem \ref{thm:QstableLSI} can be strengthened to say that, for any $t\in \R$, the probability measure proportional to $u^t \mu$ satisfies \eqref{W1boundLSI}.  The only changes are (i) the barycenter $b$ becomes $b=Z^{-1} \sigma \cdot \int x\cdot \sigma u^t d\mu,$ where  $Z:= \int u^t d\mu$ is a normalizing constant; and (ii) the constant $C$ will depend on both $\lambda$ and $t$.  See Remark \ref{rmk:Measures_ut} for details. 
\end{rque}

An important consequence of the LSI is the   classical concentration inequality for Lipschitz functions, established via Herbst's argument:   If $\mu$ satisfies \eqref{eq:LSImu} and  $F$ is $L$-Lipschitz, then  
\begin{align} 
\int{e^{F }d\mu} \leq \exp\left(     \int F d\mu + L^2/2  \right).\label{LipschitzConcentrationIneq}
\end{align}
Equality is attained if $\mu$ splits off a standard Gaussian factor in a direction $\sigma \in \mathbb{S}^{n-1}$, in which case $F(x) = L \sigma \cdot x $ achieves equality.  
The following provides a quantitative stability estimate for this result, provided $\mu$ is uniformly log-concave. 
\begin{thm}\label{thm:ThmExpConcentration}
Let  $\mu= e^{-V}dx$ be a probability measure on $\R^n$   satisfying $\Hess V  \geq \mathrm{I}_n$, and fix any  $L>0$.  There exists a constant $C(L)$ such that if  
$F: \R^n \longrightarrow \R$ satisfies $ \|F \|_{\mathrm{Lip}} \leq L$ and 
$$
\int{e^{F }d\mu} \geq \exp\left(    \int F d\mu +  \frac{L^2}{2} (1 - \epsilon/2  )  \right)
$$
for some $\epsilon \geq 0$, then there is a direction $\sigma \in \mathbb{S}^{n-1}$ for which 
\begin{align}
 W_1(\mu, \gamma_{b,\sigma} \otimes \bar{\mu})\leq C(L) \sqrt{\epsilon}, \label{W1boundLipschitz}
\end{align}
where $\gamma_{b,\sigma}$ and $\bar{\mu}$ are the same as in Theorem \ref{thm:QstableLSI}.
\end{thm}

There have been some recent works on dimension-free stability for Gaussian concentration estimates \cite{BJ17, CaMa17}, which improve the bounds with reminder terms that compare the shape of level sets to hyperplane, and can be transferred to uniformly log-concave measures via the Caffarelli contraction theorem. It is unclear whether those results and ours can be compared. 

\subsection{Strategy of proof} \label{strat_proof}

The proofs of  Theorems \ref{thm_improved_dpf} and \ref{thm:QstableLSI} are based on the same broad strategy, with three main steps. To our knowledge, this way of implementing Stein's method to study stability in variational problems is new. 

The first step can be stated in a broad abstract framework. Consider a general minimization problem of the form
$$\mu \longrightarrow \underset{f}{\inf} \hspace{1mm} \int{H(f, \nabla f)d\mu}$$
and assume the infimum over a class of probability measures $\mathcal{P}$ is known, say equal to zero, and that we can describe the subset of measures $\mu_0$ and associated functions $f_0$ such that $\int{H(f_0, \nabla f_0)d\mu_0} = 0$. Beyond the questions considered in this work, many relevant inequalities from analysis, geometry and probability can be cast in this form, such as sharp constants for Sobolev inequalities, variational problems in statistical physics, eigenvalue problems, and so on.

The Euler-Lagrange equation for problems of this form is
$$\int{u\partial_1 H(f_0, \nabla f_0) + \nabla u \cdot \partial_2 H(f_0, \nabla f_0)d\mu} = 0 \hspace{2mm} \forall u.$$
So any minimization problem of this form gives rise to an integration by parts formula for measures that achieve the infimum. Now, if we consider a measure $\mu_1$ and a function $f_1$ such that $\int{H(f_1, \nabla f_1)d\mu_1} \leq \epsilon$, the problems we consider in this paper can be stated as trying to show that $\mu_1$ is close to the class of measures at  which the infimum is reached. At a heuristic level, and maybe under extra assumptions on $f_1$, we expect an approximate Euler-Lagrange equation of the form
$$\int{u\partial_1 H(f_1, \nabla f_1) + \nabla u \cdot \partial_2 H(f_1, \nabla f_1)d\mu_1} = o(1)\times F(||f||, ||u||')$$
to hold for some class of test functions, and norms $\|\cdot\|, \| \cdot\|'$ adapted to the problem. It is in this way that we obtain an approximate integration by parts identity, which is the basic setup required for Stein's method. 

The second step is to show that $f_1$ can be replaced up to small error by a function $f_0$ such that $\int{H(f_0, \nabla f_0)d\mu_0} = 0$ for some other probability measure. In the present paper,  this is done by considering a transport map $T$ sending $\mu_0$ onto $\mu_1$, and proving that $f_1 \circ T$ approximately reaches the infimum when integrating with respect to $\mu_0$. If the minimization problem with fixed reference measure $\mu_0$ satisfies some quantitative stability property, this would mean $f_1 \circ T$ is close to some function $f_0$ for which the infimum is reached. We then deduce that $f_0$ is close to $f_1$ using specific regularity properties of the transport map, using the convexity assumptions in our problem. This part of the proof seems to be of less general scope than the other two steps. {As a tool, we use stability estimates for the sharp functional inequality with fixed reference measure.} The conclusion is that $\mu_1$ satisfies an approximate integration by parts formula
$$\int{u\partial_1 H(f_0, \nabla f_0) + \nabla u \cdot \partial_2 H(f_0, \nabla f_0)d\mu_1} = o(1)\times F(||u||').$$

The third part of the proof is to compare $\mu_1$ to $\mu_0$ using Stein's method \cite{Ste72, Ste86} and the fact that they both satisfy the same integration by parts formula, up to small error. In our situation, $\mu_0$ is Gaussian in some direction and Stein's method for such measures has been well-explored.  We expect this type of argument to also apply to non-Gaussian situations, where Stein's method has found some successes for other types of problems \cite{Ros11}. 

\section{Stability of the Poincar\'e inequality}

\subsection{Proof of Theorem \ref{thm_improved_dpf}}

First, let us note that the assumptions constrain the value of the Poincar\'e constant

\begin{lem}
Under the assumptions of Theorem \ref{thm_dpf}, we have $1 - \epsilon \leq C_P(\mu) \leq 1$. 
\end{lem}

\begin{proof}
The bound $C_P(\mu) \leq 1$ is true under the uniform convexity assumption of the potential. This is a classical result on Poincar\'e inequalities, which can be obtained for example via the \BE theory \cite{BGL14}, or the Caffarelli contraction theorem \cite{Caf00}. 
The second bound comes from 
$$1 = \int{u_1^2d\mu} \leq C_P(\mu) \int{|\nabla u_1|^2d\mu} \leq C_P(\mu) (1+\epsilon)$$
so that $C_P(\mu) \geq (1+\epsilon)^{-1} \geq 1 - \epsilon$. 
\end{proof}

We have the following bounds on proximity between the $\nabla u_i$ and unit vectors, essentially proved in \cite{DPF}: 

\begin{lem} \label{lem_dpf}
Assume $\epsilon < (18 k)^{-2}$.
Under the assumptions of Theorem \ref{thm_dpf}, there exist unit vectors $\hat{w}_1,..,\hat{w}_k \subset \R^n$ such that
$$\int{|\nabla u_i - \hat{w}_i|^2d\mu} \leq 9 \epsilon; \hspace{1cm} |\hat{w}_i \cdot \hat{w}_j|  \leq 18 \sqrt{\epsilon}, ~~~i\neq j. $$
Moreover,  $\operatorname{dim}(\operatorname{Span}(\hat{w}_1,..,\hat{w}_k))=k$.
\end{lem}
{ In particular, this lemma implies that the functions $u_i$ are close to orthogonal affine functions.}

\begin{proof}
The proof follows the arguments of \cite{DPF}, we include it for the sake of completeness. 

First, let $T$ be the optimal transport (or Brenier map) \cite{Bre91} sending the standard Gaussian measure onto $\mu$, and define $v_i:= u_i \circ T$. According to the Caffarelli contraction theorem \cite{Caf00}, $\nabla T$ is a symmetric, positive matrix satisfying $\| \nabla T\|_{op} \leq 1$. We then have
$$\int{|\nabla v_i|^2d\gamma} \leq \int{|\nabla u_i|^2 \circ T d\gamma} = \int{|\nabla u_i|^2d\mu}$$
and
$$\int{|\nabla u_i|^2d\mu} \leq (1+\epsilon)\int{u_i^2d\mu} = (1+\epsilon)\int{v_i^2d\gamma} \leq (1+\epsilon)\int{|\nabla v_i|^2d\gamma}.$$
Hence
\begin{equation}
0 \leq \int{(|\nabla u_i|^2 \circ T -|\nabla v_i|^2)d\gamma} \leq \epsilon(1+\epsilon).
\end{equation}
Additionally, since $(\operatorname{I} - \nabla T)^2 \leq \operatorname{I} - (\nabla T)^2$, we have
\begin{align}
\int{|\nabla u_i \circ T - \nabla v_i|^2d\gamma} &= \int{|(\operatorname{I} - \nabla T) (\nabla u_i \circ T)|^2d\gamma} \notag \\
&\leq \int{|\nabla u_i \circ T|^2d\gamma} - \int{|\nabla T ( \nabla u_i \circ T ) |^2d\gamma} \notag \\
&= \int{(|\nabla u_i|^2 \circ T -|\nabla v_i|^2)d\gamma} \leq \epsilon(1+\epsilon).
\end{align}

Note that $\int v_i d\gamma = \int u_i d\mu = 0$.   Since the multivariate Hermite polynomials form an orthogonal basis for $L^2( \gamma)$, we may write
 $$v_i(x) =  w_i \cdot x +  z_i(x),$$
where $w_i\in \mathbb{R}^n$ and  $z_i : \R^n \longrightarrow \R$, satisfying $\int z_i  d\gamma = 0$ and $\int z_i x_j d\gamma =0$ for $j=1, \dots, n$.   Using basic properties of Hermite polynomials,
$$1 + \epsilon \geq \int{|\nabla v_i|^2d\gamma} = |w_i|^2 + \int |\nabla z_i|^2 d\gamma \geq  1+ \frac{1}{2}\int |\nabla z_i|^2 d\gamma.$$
The second inequality is a refinement of the Gaussian Poincar\'e inequality for functions orthogonal to the subspace spanned by constant and linear functions, combined with   $ |w_i|^2 +  \int z_i^2 d\gamma = \int v_i^2 d\gamma = \int u_i^2 d\mu = 1$.  Hence, 
$$\int z_i^2 d\gamma\leq \frac{1}{2}\int |\nabla z_i|^2 d\gamma \leq  \epsilon.$$
In particular, $\int{|\nabla v_i - w_i|^2 d\gamma} \leq 2\epsilon$ and $1-\epsilon\leq |w_i|^2\leq 1$.  Together with the previous estimates,  we have for $\hat w_i := w_i/|w_i|$, 
\begin{equation*}
\int |\nabla u_i - \hat{w}_i|^2 d\mu \leq 3\left(  \int |\nabla u_i\circ T - \nabla v_i|^2 d\gamma+ \int |\nabla v_i-w_i|^2 d\gamma  +  |w_i-\hat{w}_i|^2 \right)\leq 9 \epsilon.
\end{equation*}
As a consequence, for $j\neq i$, we have
$$
|\hat{w}_i \cdot \hat{w}_j| \leq 9\epsilon + 6 \sqrt{ \epsilon(1+\epsilon)} \leq 18 \sqrt{\epsilon}.
$$
Finally, the matrix with coefficients $(\hat{w}_i \cdot \hat{w}_j)_{i,j \leq k}$ is strictly diagonally dominant when $\epsilon < (18 k)^{-2}$, and hence invertible. Thus, $\operatorname{dim}(\operatorname{Span}(\hat{w}_1,..,\hat{w}_k))=k$ as claimed.  
\end{proof}

The starting point to implement Stein's method is the following approximate integration by parts formula for the measure $\mu$ and the approximate minimizers $u_i$: 

\begin{lem} \label{lem_approx_euler}
Let $\mu$ be a probability measure satisfying a Poincar\'e inequality with constant $C_P \leq 1$. For any function $h \in H^1(\mu)$ and function $u$ satisfying $\int{ud\mu} = 0$, $\int{u^2d\mu} = 1$ and $\int{|\nabla u|^2d\mu} \leq 1 + \epsilon$, for some $\epsilon\geq 0$.  We have
$$\int{u h d\mu} - \int{\nabla u \cdot \nabla h d\mu} \leq \sqrt{\epsilon}\left(\int{|\nabla h|^2d\mu}\right)^{1/2}.$$

In particular, this applies for $u_i$ and $\mu$ satisfying the assumptions of Theorem \ref{thm_dpf}.
\end{lem}

\begin{proof}
The proof of the lemma is a variant of the argument used in \cite{CFP18, Cou18} to establish integration by parts formula mimicking the Stein identity for measures satisfying a Poincar\'e inequality. 

For any $h : \R \longrightarrow \R$ in the weighted Sobolev space $H^1(\mu)$, we have
$$\left(\int{u h d\mu}\right)^2 \leq \Var_{\mu}(h )\int{u^2d\mu} \leq C_P\int{|\nabla h|^2d\mu}.$$
Hence the original term, viewed as a function of $h$, is a continuous linear form in $H^1(\mu)$, and as an application of the Lax-Milgram theorem there exists a function $g$ such that
$$\int{uh d\mu} = \int{\nabla h \cdot \nabla gd\mu} \hspace{3mm} \forall h \in H^1(\mu); \hspace{5mm} \int{|\nabla h|^2d\mu} \leq C_P.$$
In particular, note that $\int{\nabla g \cdot \nabla u d\mu} = \int{u^2d\mu} = 1$. 

Hence for any$h \in H^1(\mu)$,
$$\int{(uh - \nabla u \cdot \nabla h) d\mu} = \int{\nabla h \cdot (\nabla g - \nabla u)d\mu} \leq \left(\int{|\nabla g - \nabla u|^2d\mu}\right)^{1/2}\left(\int{|\nabla h|^2d\mu}\right)^{1/2}.$$

Finally, we can expand the square and get 
$$\int{|\nabla g - \nabla u|^2d\mu} \leq C_P -2 + 1 + \epsilon \leq \epsilon$$
which concludes the proof. 
\end{proof}

We  shall assume without loss of generality that $p = \int x d\mu = 0$.   

Assume first that $\epsilon < 1/(18 k)^2$.   Let $(\hat{w}_1,..,\hat{w}_k)$ be as in Lemma \ref{lem_dpf}, and consider any orthonormal family $(e_1,..,e_k)$ such that $\operatorname{Span}(e_1,..,e_k)=\operatorname{Span}(\hat{w}_1,..,\hat{w}_k)$.  Let $(\alpha_{ij})_{i,j\leq k}$ be real numbers such that $e_i = \sum_{j\leq k}\alpha_{ij} \hat{w}_j$.  If $k=1$, then we may take $e_1 = \hat{w}_1$.  On the other hand, if $k\geq 2$, we  use  $|\hat{w}_i\cdot\hat{w}_k|\leq 18\sqrt{\epsilon}$ for $i\neq j$ and recall  $\epsilon < 1/(18 k)^2$ to conclude that 
$$
\sum_{j\leq k}\alpha_{ij}^2 \leq (1-18\sqrt{\epsilon})^{-1} \leq  1+  \frac{18 k\sqrt{\epsilon}}{k-1} \leq 1+  \frac{1}{k-1}.
$$
Hence, we always have  $\sum_{j\leq k}\alpha_{ij}^2\leq 2$ for each $i\leq k$.

After suitable change of coordinates, we may assume without loss of generality that the vectors $(e_i)_{i\leq k}$ coincide with the  first $k$ natural basis vectors of $\R^n$.  Hence, from now on, we write $x = (y,z)$ where $y$ is the orthogonal projection of $x$ onto the vector space spanned by the $(e_i)_{i\leq k}$, with $y_i = x \cdot e_i$, and $z$ its projection onto $\operatorname{Span}(e_1,..,e_k)^{\bot}$. Let $\bar{\mu}$ be the distribution of $z$ when $x$ is distributed according to $\mu$, that is $\bar{\mu}(dz) = e^{-W(z)}dz$ with $W(z) = -\log \int_{\operatorname{Span}(e_1,..,e_k)}{e^{-V(y,z)}dy}$. As a consequence of the Pr\'ekop\`a-Leindler theorem, $W$ inherits uniform convexity from $V$, that is $\Hess W \geq \operatorname{I}_{n-k}$ (see for example \cite{BL76}).

Consider $1$-Lipschitz $f : \R^n \longrightarrow \R$; note this ensures $f$ is integrable with respect to both $\mu$ and $\gamma_k\otimes \bar{\mu}$, where $\gamma_k$ is the centered standard Gaussian measure on $\R^k$. For any $z$, there exists a function $h(\cdot, z) : \R^k \longrightarrow \R^k$ satisfying the Poisson equation  
\begin{equation} \label{eq_poisson}
f(y,z) - \int{f(s,z)d\gamma_k(s)} = y \cdot h(y,z) - \operatorname{Tr}(\nabla_yh)(y,z).
\end{equation}
In fact, as pointed out by Barbour \cite{Bar90}, as a consequence of the representation of the Ornstein-Uhlenbeck semigroup via convolution with a Gaussian kernel, the function $h$ is given by 
\begin{align*}
h_i(y,z) &= - \partial_{e_i}\int_0^1{\frac{1}{2t}\int{(f(\sqrt{t}y + \sqrt{1-t}w, z) - f(w,z))d\gamma_k(w)}dt}\\
&= -\int_0^1{\frac{1}{2\sqrt{t(1-t)}}\int{w_i f(\sqrt{t}y + \sqrt{1-t}w, z)d\gamma_k(w)}dt},
\end{align*}
where the second identity follows from the Gaussian integration by parts formula. Hence, by the Jensen and Cauchy-Schwarz inequalities, 
\begin{align*}
|\nabla_y h_i(y,z) |^2 &= \left| \int_0^1{\frac{1}{2\sqrt{1-t}}\int{w_i \nabla_y f(\sqrt{t}y + \sqrt{1-t}w, z)  d\gamma_k(w)}dt} \right|^2\\
&\leq \int_0^1 \frac{1}{2\sqrt{1-t} } \left| \int{w_i \nabla_y f(\sqrt{t}y + \sqrt{1-t}w, z)  d\gamma_k(w)}  \right|^2 dt\\
&\leq \int_0^1 \frac{1}{\sqrt{t(1-t)} }  \int  \left|\nabla_y f(\sqrt{t}y + \sqrt{1-t}w, z) \right|^2 d\gamma_k(w)      dt.
\end{align*}
Similarly, 
\begin{align*}
|\nabla_z h_i(y,z) |^2 &= \left| \int_0^1{\frac{1}{2\sqrt{t(1-t)}}\int{w_i  \nabla_z f(\sqrt{t}y + \sqrt{1-t}w, z)   d\gamma_k(w)}dt} \right|^2\\
&\leq \int_0^1 \frac{\pi  }{4\sqrt{t(1-t)} } \left| \int{w_i  \nabla_z f(\sqrt{t}y + \sqrt{1-t}w, z)   d\gamma_k(w)}  \right|^2 dt\\
&\leq \int_0^1 \frac{ 1  }{\sqrt{t(1-t)} }  \int  \left|\nabla_z f(\sqrt{t}y + \sqrt{1-t}w, z) \right|^2 d\gamma_k(w)      dt.
\end{align*}
Combining the above, we have
\begin{align}
|\nabla h_i(y,z) |^2 = |\nabla_y h_i(y,z) |^2+|\nabla_z h_i(y,z) |^2\leq \int_0^1 \frac{ 1 }{\sqrt{t(1-t)} }  \| f\|^2_{\mathrm{Lip}}      dt\leq \pi . \label{nablaHbound}
\end{align}
The above computation follows the strategy of \cite{CM08, Gau16}. Better regularity bounds on solutions of the Poisson equation have been derived in \cite{GMS, FSX18}, but for our purpose this bound will suffice. 
It follows that $h_i \in H^1(\mu)$, justifying the following manipulations:
\begin{align*}
\int{fd\mu} - \int{fd\gamma_k d\bar{\mu}} &= \int{ \Big( y \cdot h(y,z) - \operatorname{Tr}(\nabla_y h)(y,z)\Big) d\mu} \\
&=  \sum_{i\leq k} \int \left(y_i  h_i(y,z) - e_i \cdot \nabla h_i(y,z) \right) d\mu.
\end{align*}
Now, focusing on the $i$th term in the sum, we expand
\begin{align*}
\int &\left(y_i  h_i(y,z) - e_i \cdot \nabla h_i(y,z) \right) d\mu = \sum_{j\leq k} \alpha_{ij} \int{(\nabla u_j-\hat{w}_j)\cdot \nabla h_i(y,z)d\mu} \\
&\hspace{1cm}+ \sum_{j\leq k} \alpha_{ij} \int{(\hat{w_j}\cdot x-u_j)   h_i(y,z)d\mu} + \sum_{j\leq k} \alpha_{ij} \int{\left(u_j  h_i(y,z) - \nabla u_j \cdot \nabla h_i(y,z) \right)d\mu}. 
\end{align*}
We bound each of the three terms separately.  By Cauchy-Schwarz, Lemma \ref{lem_dpf}, and \eqref{nablaHbound}
\begin{align*}
&\sum_{j\leq k} \alpha_{ij} \int{(\nabla u_j-\hat{w}_j)\cdot \nabla h_i(y,z)d\mu}\\
&\leq \left( \sum_{j\leq k} \alpha_{ij}^2 \right)^{1/2} \left( \sum_{j\leq k } \left( \int{(\nabla u_j-\hat{w}_j)\cdot \nabla h_i(y,z)d\mu}\right)^2 \right)^{1/2} \leq \sqrt{2} \left( k  \pi 9 \epsilon  \right)^{1/2}.
\end{align*}
Similarly, with additional help from the Poincar\'e inequality for $\mu$ and the assumption that $\int x d\mu =\int u_i d\mu=0$,  
\begin{align*}
 \sum_{j\leq k} \alpha_{ij} &\int{(\hat{w_j}\cdot x-u_j)   h_i(y,z)d\mu} 
 \leq \left( \sum_{j\leq k} \alpha_{ij}^2 \right)^{1/2} \left( \sum_{j\leq k } \left(  \int{(\hat{w_j}\cdot x-u_j)   h_i(y,z)d\mu}\right)^2 \right)^{1/2}\\
 &\leq  \sqrt{2} \left( \sum_{j\leq k } \left(  \int|\hat{w_j}-\nabla u_j|^2 d\mu \right) \left( \int | \nabla h_i(y,z)|^2 d\mu\right) \right)^{1/2} \leq  \sqrt{2} \left( k  \pi 9 \epsilon  \right)^{1/2}.
\end{align*}
Finally, by Lemma \ref{lem_approx_euler} and \eqref{nablaHbound}, 
\begin{align*}
&\sum_{j\leq k} \alpha_{ij} \int{\left(u_j  h_i(y,z) - \nabla u_j \cdot \nabla h_i(y,z) \right)d\mu} \\
&\leq \left( \sum_{j\leq k} \alpha_{ij}^2 \right)^{1/2} \left( \sum_{j\leq k } \left(  \int{\left(u_j  h_i(y,z) - \nabla u_j \cdot \nabla h_i(y,z) \right)d\mu} \right)^2 \right)^{1/2} \leq  \sqrt{2} \left( k \pi \epsilon  \right)^{1/2}.
\end{align*}
Combining all of the above estimates with  the Kantorovitch dual formulation of $W_1$ \cite{Vill03}, we have 
$$
W_1(\mu, \gamma_k \otimes \bar{\mu}) = \underset{f  : \|f \|_{\mathrm{Lip}\leq 1}}{\sup} \hspace{1mm} \int{fd\mu} - \int{fd\gamma_k d\bar{\mu}}
  \leq k^{3/2} 7  \sqrt{2  \pi  \epsilon} <  k^{3/2} 18  \sqrt{2   \epsilon} .
$$

To finish the proof, we only need to consider   $\epsilon \geq (18 k)^{-2}$.  In this case, we bypass Lemma \ref{lem_dpf} and   take $(e_1, \dots, e_k)$ to be any orthonormal family in $\R^n$, and define $\bar{\mu}$ in terms of this family, same as above. By the Poincar\'e inequality,  $\Var_{\mu}(x\cdot e_i)\leq 1$ for each $i\leq k$, so it follows that 
$$W_1(\mu, \gamma_k \otimes \bar{\mu}) \leq W_2(\mu, \gamma_k \otimes \bar{\mu}) \leq \sqrt{2 k}  \leq k^{3/2} 18  \sqrt{2   \epsilon} ,$$
where the last inequality holds under the assumption that $\epsilon \geq (18 k)^{-2}$.

\subsection{Proof of Theorem \ref{thm_coord}}

The proof is essentially the same as for Theorem \ref{thm_improved_dpf}, except that our extra assumptions make Lemma \ref{lem_dpf} unnecessary, which allows us to drop the convexity assumption. Without loss of generality, we may assume $\mu$ has its barycenter at the origin. We then take $u_i = \frac{x \cdot e_i}{\sqrt{\Var_{\mu}(x \cdot e_i)}}$ in Lemma \ref{lem_approx_euler} to get
$$\int{x_i h(x) - \partial_i h(x)d\mu} \leq \sqrt{ {\epsilon} }\left(\int{|\nabla h|^2d\mu}\right)^{1/2}$$
for any real-valued smooth test function $h$. We can then introduce the same function $h$ associated to a $1$-Lipschitz function $f$ via the Poisson equation \eqref{eq_poisson}, and the proof continues in the same way as the proof of Theorem \ref{thm_improved_dpf}, but is simpler since we directly conclude:
\begin{align*}
\int{fd\mu} - \int{fd\gamma_k d\bar{\mu}} &= \int{ \Big( y \cdot h(y,z) - \operatorname{Tr}(\nabla_y h)(y,z)\Big) d\mu} \\
&=   \underset{i \leq k}{\sum} \int{ (y_i h_i(y,z) -  \partial_i h_i(y,z))d\mu} \leq k  \sqrt{ \pi \epsilon}.
\end{align*}
Note that bypassing Lemma \ref{lem_dpf} gives improved  dependence on $k$. 

\section{Stability for the logarithmic Sobolev inequality}

\subsection{Proof of Theorem \ref{thm:QstableLSI}}

The proof of Theorem \ref{thm:QstableLSI} follows the strategy for that of Theorem \ref{thm_improved_dpf}, relying on an approximate integration by parts identity for extremizers of the LSI, combined with Stein's method.  However, the details are sufficiently  different that the same argument can not be applied mutatis mutandis.  The following sequence of lemmas provides the necessary ingredients for the proof. 

The following approximate Euler-Lagrange equation for the LSI is the starting point of the proof. It is used as the counterpart of Lemma \ref{lem_approx_euler}. 

\begin{lem}\label{lem:approxELeqn}  Assume $\mu$ satisfies the LSI \eqref{eq:LSImu}, and let $u:\R^n \longrightarrow \R$ satisfy \eqref{eqn:fApproxExLSI} for some $\epsilon\geq 0$. 
For any smooth function $h$ we have
\begin{align*}
& \int{\nabla h \cdot \nabla u d\mu} - \frac{1}{2}\int{h u\log(u^2/\alpha)}d\mu   \\
&\hspace{1cm}\leq \sqrt{\epsilon}\left(\int{|\nabla u|^2d\mu} \right)^{1/2} \left(\int{|\nabla h |^2d\mu} -   \frac{1}{2}\int h^2 \log(u^2/\alpha) d\mu \right)^{1/2} ,
\end{align*}
where  $\alpha := \int{u^2d\mu}$.   
\end{lem}

\begin{rque}
The quantity $\int{|\nabla h|^2d\mu} -   \frac{1}{2}\int h^2 \log(u^2/\alpha) d\mu $ is nonnegative.  Indeed, by the LSI for $\mu$, this quantity is at least
$$
\frac{1}{2}\Ent_{\mu}(h^2) -   \frac{1}{2}\int h^2 \log(u^2/\alpha) d\mu = \frac{1}{2}\int h^2 \log \left( \frac{h^2 / \int h^2 d\mu }{ u^2 / \int u^2 d\mu} \right)  d\mu,
$$
which is proportional to a relative entropy, and therefore nonnegative. 
\end{rque}

\begin{rque}
We emphasize that Lemma \ref{lem:approxELeqn} does not make any convexity assumptions on $\mu$, so may be of independent interest for other applications. 
\end{rque}

\begin{proof}
By convexity of the map $\varphi \longmapsto \Ent_{\mu}(\varphi)$ on nonnegative functions, for $t \geq 0$, it holds that 
\begin{align}
\Ent_{\mu}(\varphi + t \psi) \geq \Ent_{\mu}(\varphi ) + t \int \psi \log \left(\frac{\varphi}{\int \varphi d\mu } \right) d\mu\label{eq:EntLB}
\end{align}
provided $\varphi \geq 0$ and $\varphi + t \psi \geq 0$. 
 
Now, we observe
\begin{align*}
&2 \int |\nabla u|^2 d\mu + 4 t \int \nabla u \cdot \nabla h d\mu + 2 t \int |\nabla h |^2 d\mu \geq \Ent_{\mu}( (u+ t h)^2 ) \\
&\hspace{1cm} \geq \Ent_{\mu}( u^2 ) + t \int (2 u h + t h^2 ) \log \left(u^2/\alpha  \right) d\mu\\
&\hspace{1cm}\geq  2(1-\epsilon)\int |\nabla u|^2 d\mu + t \int (2 uh + t h^2 ) \log \left(u^2/\alpha  \right) d\mu,
\end{align*}
where the first inequality is the LSI for $\mu$ applied  to the function $u + t h$, the second inequality is \eqref{eq:EntLB}, and the third inequality is \eqref{eqn:fApproxExLSI}.  
Rearranging  and dividing by $2 t$ gives 
\begin{align*}
& \epsilon \, t^{-1} \int |\nabla u |^2 d\mu +  t  \left( \int |\nabla h |^2 d\mu   -\frac{1}{2} \int h^2  \log \left(u^2/\alpha  \right) d\mu \right)  \\
 &\hspace{1cm}\geq     \int u h  \log \left(u^2/\alpha  \right) d\mu - 2 \int \nabla u \cdot \nabla h d\mu.
\end{align*}
Optimizing over $t>0$ gives 
\begin{align*}
 &\frac{1}{2} \int u h  \log \left(u^2/\alpha  \right) d\mu -  \int \nabla u \cdot \nabla h d\mu \\
 &\hspace{1cm}\leq 
  \sqrt{\epsilon}\left(\int{|\nabla u |^2d\mu} \right)^{1/2} \left(\int{|\nabla h|^2d\mu} -   \frac{1}{2}\int h^2 \log(u^2/\alpha) d\mu \right)^{1/2}.
\end{align*}
We may now replace $h$ with $-h$ to obtain the desired inequality. 
\end{proof}

We now state the Aida-Shigekawa perturbation theorem for the LSI, which will be needed in the sequel. It will allow us to estimate certain terms that involve an extra weight $u^2$, using the fact that $\log u$ is Lipschitz. The following is a consequence of \cite[Theorem 3.4]{AS}:   
\begin{thm}\label{thm:AS}
Let $\mu$ satisfy \eqref{eq:LSImu}, and take $\mu_F$ to be the probability measure  proportional to $e^F \mu$, where $F$ is $\lambda$-Lipschitz.  There exists a $\tilde{\lambda}>0$, depending only on $\lambda$, for which 
$$
\Ent_{\mu_F}(f^2) \leq 2\tilde{\lambda}\int |\nabla f|^2 d\mu_F.
$$
In particular, $\mu_F$ also satisfies a Poincar\'e inequality with constant $C_P(\lambda)\leq \tilde{\lambda}$. 
\end{thm}
Together with \cite[Theorem 1]{FIL}, this yields the following deficit estimate for the Gaussian LSI:
\begin{lem}\label{lem:FIL}
Let the notation of Theorem \ref{thm:AS} prevail.  If $\gamma$ is the standard Gaussian measure on $\R^n$, and $d\mu_F = v^2 d\gamma$, there is a constant $c(\lambda)<1$ for which
$$
\Ent_{\gamma}(v^2) \leq 2 c(\lambda)  \int |\nabla v |^2 d\gamma.
$$
\end{lem}

The following specializes Lemma \ref{lem:approxELeqn} under the hypothesis that $\log u$ is $\lambda$-Lipschitz.  
\begin{lem}\label{lem:ELforg}
Let $u$, $\lambda$, $\epsilon$,  and $\mu$ satisfy the assumptions of Theorem \ref{thm:QstableLSI}.  If $g:\R^n\to \R$ is Lipschitz, satisfying $\int g d\mu = 0$, then  
$$
\int \nabla g \cdot \nabla \log(u) d\mu - \int g |\nabla \log u|^2 d\mu -  \int g \log u \,  d\mu  \leq \|g\|_{\mathrm{Lip}} C(\lambda) \sqrt{\epsilon}, 
$$
where $C(\lambda)$ is a constant depending only on $\lambda$. 
\end{lem}
\begin{proof}
We may assume without loss of generality that $\|g\|_{\mathrm{Lip}}\leq 1$. 

Apply Lemma \ref{lem:approxELeqn} to the test function $h=g/u$.  This gives
\begin{align}
&\int \nabla g \cdot \nabla \log(u) d\mu - \int g |\nabla \log u|^2 d\mu -  \int g \log u \,  d\mu \notag\\
&\leq \sqrt{\epsilon}\left(\int{|\nabla u|^2d\mu} \right)^{1/2} \left(2 \int{|\nabla g |^2u^{-2} d\mu}+ 2 \int{g^2 |\nabla \log u |^2 d\mu} -   \frac{1}{2}\int (g/u)^2 \log(u^2) d\mu \right)^{1/2} \notag\\
&\leq \sqrt{\epsilon}\lambda \left(2 \int{ u^{-2} d\mu}+ 2 \lambda^2 \int{g^2   d\mu} -   \frac{1}{2}\int (g/u)^2 \log(u^2) d\mu \right)^{1/2} \label{lastLine}
\end{align}
Now, we claim that for any smooth enough $h$,  we have
\begin{align}
-\frac{1}{2}\int{h^2 \log u^2 d\mu}\leq \int{|\nabla h|^2d\mu} + 2\lambda^2\int{h^2d\mu}.\label{Hbound}
\end{align}
 From the classical entropy inequality, we have
\begin{align*}
-\int{h^2 \log u^2 d\mu} &\leq \Ent_{\mu}(h^2) + \int{h^2d\mu} \times \log \int{e^{-2\log u}d\mu}\\
&\leq 2\int{|\nabla h|^2d\mu} + \int{h^2d\mu} \times \log \int{e^{-2\log u}d\mu}.
\end{align*}
Now, we apply the concentration inequality \eqref{LipschitzConcentrationIneq}.
In particular, since $\log u$ is assumed to be $\lambda$-Lipschitz
$$1 = \int{u^2d\mu} \leq \exp\left(2\int{\log u d\mu} + 2\lambda^2\right) ~~~\Longrightarrow  ~~~ - \int{\log u \, d\mu}\leq \lambda^2.$$
On the other hand, using this together with \eqref{LipschitzConcentrationIneq} gives, for any $t >0$, 
\begin{align}
\int u ^{-t}d\mu = \int{e^{-t \log u}d\mu} \leq \exp\left(-t\int{\log u \, d\mu} + (t\lambda)^2/2\right) \leq e^{t \lambda^2(1+t/2)},\label{eq:ExpLogf}
\end{align}
which leads to \eqref{Hbound} by taking $t=2$. 

Applying these estimates to \eqref{lastLine} gives 
\begin{align*}
&\int \nabla g \cdot \nabla \log(u) d\mu - \int g |\nabla \log u|^2 d\mu -  \int g \log u \,  d\mu  \\
&\leq \sqrt{\epsilon}\lambda \left(2 e^{4 \lambda^2} + 2 \lambda^2 \int{g^2   d\mu} 
+\int{|\nabla (g/u)|^2d\mu} + 2\lambda^2\int{g^2u^{-2} d\mu}
\right)^{1/2}\\
&\leq \sqrt{\epsilon}\lambda \left(4 e^{4 \lambda^2} + 2 \lambda^2   + 4\lambda^2\int{g^2u^{-2} d\mu}
\right)^{1/2},
\end{align*}
where the last line made use of the Poincar\'e inequality $\int{g^2   d\mu} \leq \int{|\nabla g|^2   d\mu}\leq 1$.  Now, since $\log u$ is $\lambda$-Lipschitz, the measure $u^{-2} \mu$ satisfies a Poincar\'e inequality with constant $C_P(\lambda)$.  Hence, 
\begin{align*}
\int g^2u^{-2} d\mu &\leq C_P(\lambda)\int |\nabla g|^2 u^{-2}d\mu + \left( \int g u^{-2} d\mu\right)^2 \left(  \int u^{-2} d\mu \right)^{-1}\\
&\leq C_P(\lambda) e^{4\lambda^2} + \left( \int g u^{-2} d\mu\right)^2 .
\end{align*}
By Cauchy-Schwarz, the Poincar\'e inequality for $\mu$ and \eqref{eq:ExpLogf}, we have
$$
 \left( \int g u^{-2} d\mu\right)^2 \leq \int g^2 d\mu \times \int u^{-4} d\mu \leq e^{12 \lambda^2},
$$
which completes the proof. 
\end{proof}

The next lemma quantifies the proximity between $\log u$ and an affine function. It is used as the counterpart to Lemma \ref{lem_dpf}. { This step is more complicated than for the Poincar\'e inequality, since in this case stability for the Gaussian functional inequality is a much more difficult problem, as we cannot simply use a spectral decomposition of the function. }

\begin{lem} \label{lem_prox_f_linear}
Let $u$, $\lambda$, $\epsilon$,  and $\mu$ be as in Theorem \ref{thm:QstableLSI}.  There exists $p\in \R^n$ and constants $C_1(\lambda)$ and $C_2(\lambda)$, depending only on $\lambda$,  such that 
\begin{align}
\int{|\nabla \log u - p/2|^2u^2d\mu} &\leq C_1(\lambda) \epsilon\int{|\nabla u|^2d\mu};\label{L2nabla_f2mu} \\
\Var_{u^2\mu}(\log u - p\cdot x/2) &\leq C_2(\lambda) \epsilon\int{|\nabla u|^2d\mu}.\label{Var_f2mu}
\end{align}
%and $|p|^2 \leq 4 \int |\nabla u|^2 d\mu$.
\end{lem}
\begin{proof}
Let $T$ be the optimal transport map sending the standard Gaussian measure onto $\mu$ and define 
\begin{align}
p := \int{\xi u( T(\xi))^2d\gamma(\xi)} = 2\int u (T(\xi)) \nabla T(\xi) \nabla u (T(\xi)) d\gamma(\xi), \label{pDef}
\end{align}
where the second identity follows from   Gaussian integration by parts.  The Caffarelli contraction theorem states that $T$ is $1$-Lipschitz. Define $v(x) = u(T(x+p) )e^{-p\cdot x/2 - |p|^2/4}$. We have
$$\int{v^2d\gamma} = \int{u (T(x+p))^2e^{-|x+p|^2/2}(2\pi)^{-n/2}dx} = \int{u(T(\xi))^2d\gamma(\xi)} = \int{u^2d\mu} = 1;$$
$$\int{xv^2d\gamma} = \int{x u(T(x+p))^2e^{-|x+p|^2/2}(2\pi)^{-n/2}dx} = \int{(\xi-p)u(T(\xi))^2d\gamma(\xi)} = 0.$$
Hence $v^2d\gamma$ is a centered probability measure. Moreover, since $\log u(T(x+p))$ is $\lambda$-Lipschitz, the measure $v^2d\gamma$ satisfies a Poincar\'e inequality with a constant $C_P(\lambda)$ by Theorem \ref{thm:AS}. 

We have $\Ent_{\gamma}(v^2) = \Ent_{\mu}(u^2) - |p|^2/2$ and, using the fact that $T$ is $1$-Lipschitz and the identity \eqref{pDef}, 
$$\int{|\nabla v|^2d\gamma} = \int{|\nabla T(x+p) (\nabla \log u)(T(x+p)) - p/2|^2v^2d\gamma} \leq \int{|\nabla u|^2d\mu} - |p|^2/4.$$

Hence the deficit in the Gaussian LSI for the probability measure $v^2d\gamma$ is smaller than $2\epsilon\int{|\nabla u|^2d\mu}$. By Lemma \ref{lem:FIL}, this ensures that 
\begin{align}
\int{|\nabla (\log u \circ T) - p/2|^2(u^2\circ T)d\gamma} = \int{|\nabla v|^2d\gamma} \leq C_0(\lambda)\epsilon \int{|\nabla u|^2d\mu}.\label{boundFirstTerm}
\end{align}
In a different direction, we use the Gaussian LSI to observe that 
\begin{align}
(1-\epsilon) \int |\nabla u|^2 d\mu &\leq \frac{1}{2 } \Ent_{\mu}(u^2) \notag\\
&= \frac{1}{2} \left( \Ent_{\gamma}(v^2) + |p|^2/2 \right) 
\leq  \left( \int |\nabla v|^2 d\gamma + |p|^2/4 \right).\label{revEst}
\end{align}
Now, the proof continues along similar lines to that of Lemma \ref{lem_dpf}.  First, we bound
\begin{align*}
&\frac{1}{2}\int{|\nabla \log u - p/2|^2u^2d\mu} \\
&\leq \int | \nabla \log(u\circ T)-p/2  |^2 (u^2\circ T)d\gamma  + \int |(\nabla \log u)\circ T - \nabla \log(u\circ T)|^2 (u^2\circ T)d\gamma.
\end{align*}
The first term on the RHS is controlled by \eqref{boundFirstTerm}.  The second term is bounded as 
\begin{align}
 &\int |(\nabla \log u)\circ T - \nabla \log(u\circ T)|^2 (u^2\circ T)d\gamma\notag\\
 &=\int |( I-\nabla T)(\nabla \log u)\circ T |^2 (u^2\circ T)d\gamma\notag\\
  &\leq \int | (\nabla \log u)\circ T |^2 (u^2\circ T)d\gamma - \int |\nabla T(\nabla \log u)\circ T |^2 (u^2\circ T)d\gamma \label{psdIneq}\\
    &= \int | \nabla u |^2  d\mu - \left(\int |\nabla v|^2 d\gamma + |p|^2/4 \right) \label{useVdefn} \\
    %&\leq \frac{\epsilon}{1-\epsilon } \left( \int |\nabla v|^2 d\gamma + |p|^2/4 \right) \label{}\\
    &\leq \epsilon \int{|\nabla u|^2d\mu}, \label{useRevIneq}
\end{align}
where \eqref{psdIneq} follows since $(I-\nabla T)^2 \leq I - (\nabla T)^2$,  \eqref{useVdefn} follows by definition of $v$,    and \eqref{useRevIneq} is due to  \eqref{revEst}.  This establishes \eqref{L2nabla_f2mu}.   Since $\log u$ is Lipschitz, the measure $u^2\mu$ satisfies a Poincar\'e inequality with constant depending only on $\lambda$ by Theorem \ref{thm:AS}, so that $\Var_{u^2\mu}(\log u - p\cdot x/2) \leq C_2(\lambda)\epsilon$ as desired. 

\end{proof}

Combining these estimates leads to the following approximate integration by parts formula, which is the crucial estimate we need:
\begin{lem}\label{lem:approxExpLinear}
Let $u$, $\lambda$, $\epsilon$,  and $\mu$ satisfy the assumptions of Theorem \ref{thm:QstableLSI}, and let $p\in \R^n$ be as in Lemma \ref{lem_prox_f_linear}.  For any  Lipschitz function  $g$, we have
\begin{align}
\int{\left(g \times   \left(x-\langle x\rangle_{\mu} \right) \cdot p - \nabla g \cdot p\right) d\mu} \leq   \|g\|_{\mathrm{Lip}}C(\lambda)  \sqrt{\epsilon},\label{eq:ApproxIntByParts}
\end{align}
where $C(\lambda)$ is a constant depending only on $\lambda$, and $\langle x\rangle_{\mu} := \int x d\mu$.  
\end{lem}
\begin{rque}\label{rmk:Measures_ut}
To modify Theorem \ref{thm:QstableLSI} for measures $u^t \mu$ along the lines of Remark \ref{Rmk_utdmu}, one should modify Lemma \ref{lem:ELforg} by repeating the proof mutatis mutandis, except one should consider the test function $h=g u^{t-1}$, rather than $h=g/u$.   The following proof can then be suitably modified to yield an approximate  integration by parts formula \eqref{eq:ApproxIntByParts} for the measure $u^t \mu$.  Lemma \ref{lem_prox_f_linear} does not need to be modified. 
\end{rque}

\begin{proof}
Since the statement to prove is invariant to adding a constant to $g$, we assume without loss of generality that $\int g d\mu  =0$, and that $\|g\|_{\mathrm{Lip}}\leq 1$.  Throughout, we let $C(\lambda)$ denote a constant depending only on $\lambda$ which may change line to line. 

Letting $\beta = \int (\log u - x\cdot p/2)d\mu$, we have by Cauchy-Schwarz
\begin{align*}
\frac{1}{2}\int g x\cdot p d\mu - \int g \log u\, d\mu &= \int g (\beta - \log u + x\cdot p/2) d\mu\\
&\leq \left( \int g^2 u^{-2}d\mu \right)^{1/2} \Var_{u^2 \mu}\Big( \log u -  x\cdot p/2  \Big)^{1/2}\\
&\leq  C(\lambda) \sqrt{\epsilon},
\end{align*}
where the last line follows from \eqref{Var_f2mu}, the fact that $\int |\nabla u|^2 d\mu = \int |\nabla \log u|^2 u^2 d\mu \leq \lambda^2$, and the estimate $\int g^2 u^{-2}d\mu\leq C(\lambda)$ established in the final steps of the proof of Lemma \ref{lem:ELforg}. 

Next, we write 
\begin{align*}
&\int \nabla g \cdot \nabla \log u \, du - \int g |\nabla \log u|^2 d\mu - \frac{1}{2}\int \nabla g \cdot p d\mu + \frac{1}{2} \int g\,  p \cdot \nabla \log u d\mu\\
&= \int \nabla g \cdot (\nabla \log u - p/2) d\mu  + \int g \nabla \log u \cdot (p/2 - \nabla \log u )d\mu \\
&\leq \left( \left( \int |\nabla g|^2u^{-2} d\mu \right)^{1/2} + \left( \int g^2 |\nabla u |u^{-2} d\mu \right)^{1/2}\right) \left( \int |\nabla \log u - p/2|u^2 d\mu \right)^{1/2}  \\
&\leq \left( \left( \int u^{-2} d\mu \right)^{1/2} + \lambda \left( \int g^2 u^{-2} d\mu \right)^{1/2}\right) \left( \int |\nabla \log u - p/2|u^2 d\mu \right)^{1/2}  \\
&\leq C(\lambda) \sqrt{\epsilon},
\end{align*}
where the final inequality follows similarly to before, except using \eqref{L2nabla_f2mu}.  

Summing the estimates and applying Lemma \ref{lem:ELforg}, we have
\begin{align*}
 \int \Big( g \times (x- \langle x\rangle_{\mu})\cdot p - \nabla g \cdot p \Big) d\mu +   \int g\,  p \cdot \nabla \log u d\mu  \leq C(\lambda) \sqrt{\epsilon},
\end{align*}
where the $\langle x\rangle_{\mu}$ was inserted using the assumption that $\int g d\mu = 0$.  Thus, it only remains to show that the error term is small.  To this end, we again use $\int g d\mu = 0$ to write 
\begin{align*}
\left|   \int g\,  p \cdot \nabla \log u d\mu \right| &= \left|   \int g\,  p \cdot (\nabla \log u -p/2) d\mu \right| \\
%&\leq  \int g\,  p \cdot (\nabla \log u -p/2) d\mu \\
&\leq |p| \left(\int g^2  u^{-2}d\mu \right)^{1/2}\left( \int |\nabla \log u - p/2|u^2 d\mu \right)^{1/2} \\
&\leq C(\lambda)\sqrt{\epsilon},
\end{align*}
which follows from similar estimates as above, plus the fact that $|p|^2 \leq 4 \int |\nabla u|^2 d\mu = 4 \int |\nabla \log u|u^2 d\mu \leq  4 \lambda^2$, where the first inequality was observed in the proof of Lemma \ref{lem_prox_f_linear}. 
\end{proof}

Combining this last lemma and Stein's method, we now prove  Theorem \ref{thm:QstableLSI}.
 
%\begin{thm}\label{thm:LSInuEst}
%Let $u$, $\lambda$, $\epsilon$,  and $\mu$ be as in Theorem \ref{thm:QstableLSI}.  Further, let $\nu$ be the probability measure with density $d\nu = Z^{-1} u d\mu$, where $Z=\int u d\mu$ is a normalizing constant. There exist constants $C(\lambda)$ and  $\epsilon_0(\lambda)$ depending only on $\lambda$ such that, if  $\epsilon \leq \epsilon_0(\lambda)$,  there is a direction $e \in \mathbb{S}^{n-1}$  for which
%$$
% W_1(\nu, \gamma_{b,e} \otimes \bar{\nu})\leq C(\lambda) \sqrt{\epsilon}\left(\int{|\nabla u|^2d\mu}\right)^{-1/2},
%$$
%where $\gamma_{b,e}$ denotes the standard Gaussian measure on $\operatorname{Span}(e)$ with barycenter $b=e \int (e \cdot x) d\nu$, and $\bar{\nu}$ denotes the marginal distribution of $\nu$ on  $\operatorname{Span}(e)^{\bot}$.
%\end{thm}

\begin{proof}[Proof of Theorem \ref{thm:QstableLSI}] Since the statement to prove is translation invariant, we assume $\int x d\mu=0$.   We assume first that   $\epsilon \leq  1/(4 C_1(\lambda))$, where $C_1(\lambda)$ is as defined in Lemma \ref{lem_prox_f_linear}.  The same lemma ensures existence of $p\in \R^n$ such that 
$$
\int{|\nabla \log u - p/2|^2u^2d\mu} \leq C_1(\lambda) \epsilon\int{|\nabla u|^2d\mu}.
$$
Thus, using the assumption that $\epsilon \leq  1/(4 C_1(\lambda))$, we apply the elementary inequality $|A-B|^2\geq \frac{1}{2}|A|^2-|B|^2$ to the above to conclude 
$$|p|^2 \geq  (2 - \epsilon\, 4 C_1(\lambda))\int{|\nabla u|^2d\mu}\geq    \int{|\nabla u|^2d\mu}.$$
% Ensuring that $1 - \epsilon\, C_1(\lambda) > 0$ will require to take $\epsilon$ small enough, in a way that depends on $\lambda$. 
Henceforth, we let $C(\lambda)$ denote a constant depending on $\lambda$, which may change from line to line.   The vector $p$ above is the same as in Lemma \ref{lem:approxExpLinear}, so we apply it and combine with the above estimate on $|p|$ to find for  $e := p/|p|$, 
$$\int{(g \, x\cdot e - \nabla g \cdot e)d\mu} \leq \| g\|_{\mathrm{Lip}}C(\lambda) \sqrt{\epsilon}\left(\int{|\nabla u|^2d\mu}\right)^{-1/2} ,$$
holding for any  Lipschitz  $g: \R^n \longrightarrow \R$. 

Now, we implement Stein's method following the proof of Theorem \ref{thm_improved_dpf}. In particular, we begin by writing $x = (y,z)$ where $y$ is the orthogonal projection of $x$ onto $e$, and $z$ its projection onto $e^{\bot}$.   
Consider $1$-Lipschitz $f : \R^n \longrightarrow \R$. 
For any $z\in \R^{n-1}$, there exists a function $g(\cdot, z) : \R \longrightarrow \R$ satisfying $$f(y,z) - \int_{\operatorname{Span}(e)}{f(s,z)d\gamma_{0,e}(s)} = y   g(y,z) - \partial_y g (y,z),$$
where $\gamma_{0,e}$ is the centered standard Gaussian measure on $\operatorname{Span}(e)$. 

The function $g$ is measurable and satisfies $\|g \|_{\mathrm{Lip}}\leq \sqrt{\pi}$, as already shown in \eqref{nablaHbound}.  Hence, we integrate with respect to $\mu$ to conclude
\begin{align*}
\int{f d\mu} - \int{ f d\gamma_{0,e} d\bar{\mu}} &= \int{ \Big( y   g(y,z) -  \partial_y g(y,z)\Big) d\mu} \\
&=    \int \left(  g\times   (x\cdot e) - e \cdot \nabla g  \right) d\mu  \leq C(\lambda) \sqrt{\epsilon} \left(\int{|\nabla u |^2d\mu}  \right)^{-1/2} .
\end{align*}
Since $f$ was an arbitrary 1-Lipschitz function, the theorem follows from the Kantorovich dual formulation of $W_1$, provided $\epsilon \leq  1/(4 C_1(\lambda))$.

Now, by the triangle inequality for $W_1$ and simple variance bounds, it is easy to see that 
$W_1(\mu, \gamma_{b,\sigma}\otimes \bar{\mu})\leq 2$ 
for any $\sigma \in \mathbb{S}^{n-1}$ and $b = \sigma \int x\cdot \sigma d\mu$. Hence, the $W_1$ estimate \eqref{W1boundLipschitz}   can not become active until   $ {\epsilon} \leq 4 \left( \int |\nabla u|^2 d\mu \right)  /C(\lambda)^2 \leq 4 \lambda^2 / C(\lambda)^2$.  By suitable modification of $C(\lambda)$, we may assume $C(\lambda)^2 \geq 16 \lambda^2  C_1(\lambda)$, so that the claim of the theorem is automatically satisfied whenever $\epsilon >  1/(4 C_1(\lambda))$.  This completes the proof. 
\end{proof}

\subsection{Proof of Theorem \ref{thm:ThmExpConcentration}}

\begin{proof}[Proof of Theorem \ref{thm:ThmExpConcentration}] By suitable modification, we can assume without loss of generality that $\int F d\mu =0$ and $C(L)\geq 2\sqrt{2}$, so that we may restrict attention to the case where   $\epsilon\leq 1/2$ (in the complementary case, the $W_1$ estimate will be automatically satisfied for similar reasons as argued in the final steps of the proof of Theorem \ref{thm:QstableLSI}). %, and let $L \geq  \|F \|_{\mathrm{Lip}}$.  
The Herbst argument establishes \eqref{LipschitzConcentrationIneq} by considering the function 
$$
H(\lambda) = \log\left( \int e^{\lambda F} d\mu \right) , 
$$
and using the LSI to establish  the differential inequality 
$$
\frac{d}{d\lambda}\left( \frac{H(\lambda)}{\lambda}\right) = \frac{H'(\lambda)}{\lambda} - \frac{H(\lambda)}{\lambda^2}\leq \frac{L^2}{2}.
$$
This is then integrated with respect to $\lambda$ on $(0,1)$ to establish the inequality \eqref{LipschitzConcentrationIneq}.  %(note that $H(\lambda)/\lambda \to 0$ as $\lambda\to 0^+$ by l'H\^opital's rule). 
So, by Markov's inequality, 
\begin{align*}
\left|\left\{s\in (0,1) : \frac{L^2}{2} - \left.  \frac{d}{d\lambda}\left(\frac{H(\lambda)}{\lambda} \right)\right|_{\lambda =s}  \geq \frac{\epsilon L^2}{2} \right\}\right| &\leq \frac{2}{\epsilon L^2} \left(  \frac{L^2}{2} - \log\left( \int e^{  F} d\mu \right)  \right) \leq  \frac{1}{2}, 
\end{align*}
where the last inequality follows by our hypothesis on $F$.  Therefore, there exists $\lambda_0 \in [1/2,1]$ for which 
$$
\frac{\int e^{\lambda_0 F}\log(e^{\lambda_0 F} ) d\mu }{\int e^{\lambda_0 F} d\mu   }  - \log\left( \int e^{\lambda_0 F} d\mu \right) 
=   \lambda_0 H'(\lambda_0)  -  H(\lambda_0) \geq (1-\epsilon) \lambda_0^2    \frac{L^2}{2}  .
$$
Multiplying through by $\int e^{\lambda_0 F} d\mu $, we have 
\begin{align}
\Ent_{\mu}(e^{\lambda_0 F} ) \geq (1-\epsilon) \lambda_0^2    \frac{L^2}{2}  \int e^{\lambda_0 F} d\mu  \geq 2(1-\epsilon)  \int \left| \nabla e^{\lambda_0 F/2}\right|^2 d\mu.\label{almostLSI}
\end{align}
Since $\lambda_0\leq 1$, we have that $\log e^{\lambda_0 F/2}$ is $L/2$-Lipschitz.  As a consequence,  Theorem \ref{thm:QstableLSI} applies to yield the estimate 
\begin{align}
W_1(\mu,  \gamma_{b,\sigma} \otimes \bar{\mu} )\leq C(L)    \left( \int \left|\nabla e^{\lambda_0 F/2}\right|^2 d\mu  \right)^{-1/2}\sqrt{\epsilon},\label{fromThmW1}
\end{align}
for probability measures $\gamma_{b,\sigma}, \bar{\mu}$ as defined in the statement of the theorem. 

By the LSI for $\mu$ together with \eqref{almostLSI}, we have 
$$
 \int \left| \nabla e^{\lambda_0 F/2}\right|^2 d\mu  \geq (1-\epsilon) \lambda_0^2    \frac{L^2}{4}  \int e^{\lambda_0 F} d\mu \geq (1-\epsilon)     \frac{L^2}{16}  \int e^{\lambda_0 F} d\mu .
$$
Using the fact that $\frac{d}{d\lambda}\left(\lambda \frac{L^2}{2} -  \frac{H(\lambda)}{\lambda}\right) \geq 0$, we have
$\frac{L^2}{2} - H(1) \geq \lambda_0 \frac{L^2}{2}  - \frac{H(\lambda_0)}{\lambda_0}$.
Rearranging yields, for $\epsilon < 1$, 
$$
 \int e^{\lambda_0 F} d\mu \geq e^{-\lambda_0(1-\lambda_0)L^2/2 }\left(  \int e^{ F}d\mu\right)^{\lambda_0} \geq 
 e^{-\lambda_0(1-\lambda_0)L^2/2 }\left(  e^{L^2/2(1-\epsilon/2)}  \right)^{\lambda_0}\geq e^{L^2/8}.
$$
Therefore, $\int \left| \nabla e^{\lambda_0 F/2}\right|^2 d\mu \geq    \frac{L^2}{32} e^{L^2/8}$, so that this term  can be absorbed into the constant $C(L)$ in \eqref{fromThmW1}, completing the proof. 
\end{proof}

We conclude with a stability estimate for another formulation of Gaussian concentration. The Markov inequality argument applied to \eqref{LipschitzConcentrationIneq} shows that any 1-Lipschitz $F$ satisfies the Gaussian concentration inequality 
$$
\mu\left(\left\{ F \geq t + \int F d\mu \right\} \right) \leq e^{-t^2/2}, ~~~~ t\geq 0.
$$
Unlike \eqref{LipschitzConcentrationIneq}, the inequality here is actually strict, and this form of concentration inequality is actually strictly weaker. A simple corollary of Theorem \ref{thm:ThmExpConcentration} is  the following stability version of this result. 
\begin{cor}
Let  $\mu$ and $C$ be as in Theorem \ref{thm:QstableLSI}, and consider 1-Lipschitz  $F$.  If 
$$
\mu\left(\left\{ F \geq t + \int F d\mu \right\} \right) \geq \exp\left( -(1+\epsilon/2)\frac{t^2}{2}\right)
$$
for some $t>0$ and $\epsilon \geq 0$, then $\mu$ satisfies \eqref{W1boundLipschitz} for $L=t$. 
\end{cor}

However, the classical concentration bound $\mu\left(\left\{ F \geq t + \int F d\mu \right\} \right) \leq e^{-t^2/2}$ can be sharpened into a bound of the form $\mu\left(\left\{ F \geq t + \int F d\mu \right\} \right) \leq Ce^{-t^2/2}/t$, using for example the Bakry-Ledoux isoperimetric inequality \cite{BL96} or the Caffarelli contraction theorem and refined concentration bounds for the Gaussian measure. Because the dependence of $C$ on $t$ is not explicit, it may be that the above Corollary is vacuous, in that taking $\epsilon$  small enough relative to $C(t)$ (to activate the $W_1$ estimate \eqref{W1boundLipschitz})  always makes the above lower bound greater than the improved upper bound.  As such, it is not clear if this statement is of any interest, but we include it because the question of stability for this way of encoding Gaussian concentration for uniformly log-concave measures seemed like a natural question the reader may wonder about after reading this work. 

\begin{proof}
We may assume that    $\int F d\mu = 0$.  By the hypothesis and the Markov inequality, we have 
$$
\exp\left( -(1+\epsilon/2)\frac{t^2}{2}\right)\leq \mu\left(\{ F \geq t \} \right) \leq e^{-t^2} \int e^{t F}d\mu .
$$
Multiplying through by $\exp(t^2)$ gives 
$$
\int e^{t F}d\mu \geq \exp\left( \frac{t^2}{2}(1-\epsilon/2)\right).
$$
Hence, Theorem \ref{thm:ThmExpConcentration} applies to the $t$-Lipschitz function $t F$.
\end{proof}

\vspace{5mm}

\textbf{\underline{Acknowledgments:}} This work benefited from support from the France-Berkeley fund, and ANR-11-LABX-0040-CIMI within the program ANR-11-IDEX-0002-02. M.F. was partly supported by Project EFI (ANR-17-CE40-0030) of the French National Research Agency (ANR) T.C. was partly supported by NSF grants CCF-1750430 and CCF-1704967. We thank Lorenzo Brasco and Michel Ledoux for their helpful remarks on a preliminary version of this work.

\end{document}